\documentclass[12pt, twoside, 
]{amsart} 

\title[Versions of injectivity and extension theorems]
{Versions of injectivity and extension theorems}

\author{Yoshinori  Gongyo}

\address{Graduate School of Mathematical Sciences, 
the University of Tokyo, 3-8-1 Komaba, Meguro-ku, Tokyo 153-8914, Japan.}

\email{gongyo@ms.u-tokyo.ac.jp}

\address{Department of Mathematics, Imperial College London, 180 Queen's Gate, London SW7 2AZ, UK.}
 \email{y.gongyo@imperial.ac.uk}

\author{Shin-ichi Matsumura}
\address{Kagoshima University, 1-21-35 Koorimoto, Kagoshima 890-0065, Japan.}
\email{shinichi@sci.kagoshima-u.ac.jp, mshinichi0@gmail.com}

\date{\today, version 1.06}



\newcommand{\Supp}[0]{{\operatorname{Supp}}}

\newcommand{\ddbar}{dd^c}

\newcommand{\dbar}{\overline{\partial}}
\newcommand{\e}{\varepsilon}
\newcommand{\ome}{\widetilde{\omega}}
\newcommand{\I}[1]{\mathcal{I}(#1)}

\newcommand{\lla}[0]{{\langle\!\langle}}
\newcommand{\rra}[0]{{\rangle\!\rangle}}

\usepackage{pxfonts}

\usepackage{latexsym}
\usepackage{amsmath}
\usepackage{amssymb}
\usepackage{amsthm}
\usepackage{amscd}
\usepackage{enumerate} 
\usepackage{amssymb} % \mathfrak , \mathbb
\usepackage{mathrsfs}          % \mathscr  
\usepackage[all]{xy}
\usepackage{color}

\newtheorem{thm}{Theorem}[section]

\newtheorem{prop}[thm]{Proposition}
\newtheorem{lem}[thm]{Lemma}
\newtheorem{rem}[thm]{Remark}
\newtheorem{cor}[thm]{Corollary}
\newtheorem{conj}[thm]{Conjecture}

\newtheorem{cl}[thm]{Claim}

 \theoremstyle{definition}
\newtheorem{defi}[thm]{Definition}
\newtheorem{eg}[thm]{Example}

\newtheorem*{ack}{Acknowledgments}

\begin{document}
\bibliographystyle{amsalpha+}
 
\maketitle
 
\begin{abstract}
We give an analytic version of the injectivity theorem by using multiplier ideal sheaves, 
and prove some extension theorems for the adjoint bundle of dlt pairs. 
Moreover, by combining techniques of the minimal model program,   
we obtain 
some results for semi-ampleness related to 
the abundance conjecture in birational geometry 
and the Strominger-Yau-Zaslow conjecture 
for hyperK\"ahler manifolds.

\end{abstract}

\section{Introduction}\label{real-intro}
Throughout this paper, we work over $\mathbb{C}$, the complex number field, 
and  freely use the standard notation in 
\cite{bchm}, \cite{Dem}, \cite{kamama}, and \cite{komo}. 
Further we interchangeably use the words \lq \lq Cartier divisors", \lq \lq line bundles", and 
\lq \lq invertible sheaves".

The following conjecture is one of 
the most important conjectures in the classification theory of algebraic varieties:

\begin{conj}[Generalized abundance conjecture]\label{abun conj}Let $X$ be a normal projective variety and $\Delta$ be an effective $\mathbb{Q}$-divisor such that $(X,\Delta)$ is a klt pair. 
Then  
$\kappa(X, K_X+\Delta)=\kappa_{\sigma}(X, K_X+\Delta)$.   In particular, if $K_X+\Delta$ is nef, then it is semi-ample. 
$($For the definition of $\kappa(\cdot)$ and $\kappa_{\sigma}(\cdot)$, 
see  \cite{nakayama-zariski-abun}.$)$
\end{conj}

Toward the abundance conjecture, 
we need to solve the non-vanishing conjecture  
and the extension conjecture, see \cite{dhp}, \cite[Introduction]{fujino-abundance}, \cite[Section5]{fg3}.  
One of the purposes of this paper is to study 
the following extension conjecture for the adjoint bundle of dlt pairs 
formulated in \cite[Conjecture 1.3]{dhp}:

\begin{conj}[{Extension conjecture for dlt pairs}]\label{c-dlt} 
Let $X$ be a normal projective variety 
and $S + B$ be an effective $\mathbb Q$-divisor
with the following assumptions: 
\begin{itemize}
\item $(X, S+B)$ is a dlt pair.
\item $\lfloor S+B \rfloor =S$.  
\item $K_X+S+B$ is nef. 
\item $K_X+S+B$ is $\mathbb Q$-linearly equivalent to 
an effective divisor $D$ with $S\subseteq \Supp D \subseteq \Supp\,(S+B)$. 
%$K_X+S+B\sim _{\mathbb Q}D\geq 0$ with $S\subset\Supp D$. 
\end{itemize}
Then the restriction map 
$$H^0(X,\mathcal O _X(m(K_X+S+B)))\to H^0(S,  \mathcal O _S(m(K_X+S+B)))  $$
is surjective for all sufficiently divisible integers $m\geq 2$. 
\end{conj}

When $S$ is a normal irreducible variety (namely (X, S+B) is a plt pair), 
Demailly--Hacon--P\u{a}un proved the above conjecture in \cite{dhp}
by using technical methods based on a version of the Ohsawa-Takegoshi extension theorem.
This result can be seen as an extension theorem 
for plt pairs.

In this paper,  
we study the extension conjecture for {\textit{dlt pairs}}
by giving an analytic version of the injectivity theorem 
instead of the Ohsawa-Takegoshi extension theorem. 
Thanks to the injectivity theorem,  
we can obtain some extension theorems for not only plt paris but also dlt pairs. 
This is one of our advantages. 
Our injectivity theorem is as follows: 

\begin{thm}[A version of the injectivity theorem {=Theorem \ref{main-inj}}]
\label{intro-main-inj}
Let $(F, h_{F})$ and $(L, h_{L})$ 
be $($singular$)$ hermitian line bundles 
with semi-positive curvature 
on a compact K\"ahler manifold $X$. 
Assume that there exists 
an effective $\mathbb{R}$-divisor $\Delta$ with 
\begin{align*}
h_{F}=h_{L}^{a}\cdot h_{\Delta},  
\end{align*}
where $a$ is a positive real number and  $h_{\Delta}$  is 
the singular metric defined by $\Delta$.

Then for a $($non-zero$)$ section $s$ of $L$
satisfying $\sup_{X} |s|_{h_{L}} < \infty$, 
the multiplication map 
%induced by the tensor product with $s$ 
\begin{equation*}
\Phi_{s}: H^{q}(X, K_{X} \otimes F \otimes \I{h_{F}}) 
\xrightarrow{\otimes s} 
H^{q}(X, K_{X} \otimes F\otimes L \otimes \I{h_{F} h_{L}})
\end{equation*}
is $($well-defined and$)$ injective for any $q$. 
Here $\I{h}$ denotes the multiplier ideal sheaf 
associated to a singular metric $h$.   
\end{thm}

The injectivity theorem has been 
studied by several authors, for example, Tankeev \cite{tan}, Koll\'ar \cite{Kol86}, Enoki \cite{Eno90}, Ohsawa \cite{ohsawa}, Enault--Viehweg \cite{ev}, Fujino \cite{fujino-book1}, \cite{Fuj12-A}, \cite{F-fund}, \cite{fujino-book2}, and Ambro \cite{ambro-qlog}, \cite{ambro}. 
Recently  the second author gave an injectivity theorem with multiplier ideal sheaves, 
which corresponds to the case of  $\Delta=0$ (see \cite{Mat-inj}).
Our proof is a generalization of  the methods of \cite{Eno90}, \cite{ohsawa},  
\cite{Fuj12-A}, and  \cite{Mat-inj}.

By applying 
this injectivity theorem to an extension problem, we obtain the following extension theorem. 
Even if $K_{X}+\Delta$ is semi-positive 
(namely admits a smooth metric with semi-positive curvature), 
it seems to be not able to  prove Conjecture \ref{c-dlt} for dlt pairs 
by the Ohsawa-Takegoshi extension theorem in at least current situation, 
and thus we need our injectivity theorem  (Theorem \ref{intro-main-inj}).

\begin{thm}[{=Theorem \ref{ext}}]\label{intro-ext}
Let $X$ be a compact K\"ahler manifold and 
$S+B$ be an effective $\mathbb Q$-divisor with the 
following assumptions: 
\begin{itemize}
\item $S+B$ is a simple normal crossing divisor with 
$0 \leq S+B \leq 1$ and $\lfloor S+B \rfloor =S$.
\item $K_{X}+S+B$ is ${\mathbb{Q}}$-linearly 
equivalent to an effective divisor $D$ 
with $S \subseteq\mathrm{Supp}\, D$. 
\item $K_{X}+S+B$ admits a $($singular$)$ metric $h$ with semi-positive curvature. 
\end{itemize}
Then for an integer $m \geq 2$ with Cartier divisor $m(K_{X}+S+B)$ 
and any section $u \in H^{0}(S, \mathcal{O}_{S}(m(K_{X}+S+B)))$ 
that comes from  
$H^{0}(S, \mathcal{O}_{S}(m(K_{X}+S+B))\otimes \I{h^{m-1}h_{B}} )$,  
the section $u$ can be extended to a section in $H^{0}(X, \mathcal{O}_{X}(m(K_{X}+S+B)) )$. 
\end{thm}

In particular, we can prove Conjecture \ref{c-dlt} 
under the assumption that $K_{X}+ \Delta$ admits a (singular) metric 
whose curvature is semi-positive and Lelong number is identically zero 
(see Corollary \ref{ext-cor}). 
This assumption is stronger than the assumption that $K_{X}+\Delta$ is nef, but 
weaker than the assumption that $K_{X}+\Delta$ is semi-positive. 
Remark that Verbitsky proved the non-vanishing conjecture on hyperK\"ahler manifolds  
under the same assumption (see \cite{ver}).

As compared with Conjecture \ref{c-dlt}, one of our advantages is removing the condition  $\Supp\,D\subseteq \Supp (S+B) $ (this version of the extension conjecture  is  \cite[Conjecture 5.8]{fg3}). 
Thanks to removing the condition  $\Supp\,D\subseteq \Supp (S+B) $ in Conjecture \ref{c-dlt}, we can run the MMP preserving the good condition for metrics 
(cf. \cite[Section 8]{dhp} and \cite[Theorem 5.9]{fg3}). 
By applying the above theorem and techniques of 
the MMP, we obtain the following theorem related to 
the abundance conjecture:

\begin{thm}[{=Theorem \ref{main-thm-semiample}}]\label{main-thm}
Assume that Conjecture \ref{abun conj} holds in dimension $n-1$.
Let $X$ be an $n$-dimensional normal projective variety and $\Delta$ be an $\mathbb{Q}$-divisor 
with the following assumptions: 
\begin{itemize}
\item $(X,\Delta)$ is a klt pair. 
\item  There exists a projective birational morphism $\varphi:Y \to X$ such that $Y$ is smooth and $\varphi^*(\mathcal{O}_X(m(K_X+\Delta)))$ admits 
a $($singular$)$ metric whose curvature 
is semi-positive and Lelong number is identically zero.
Here $m$ is a positive integer such that $m(K_X+\Delta)$ is Cartier.  
\end{itemize}
If $\kappa(K_X+\Delta) \geq 0$, then $K_X+\Delta$ is semi-ample.

\end{thm}

Finally, by combining Verbitsky's non-vanishing theorem 
(\cite[Theorem 4.1]{ver})
on hyperK\"ahler manifolds 
(holomorphic symplectic manifolds), 
we obtain a result for semi-ampleness on four dimensional projective hyperK\"ahler manifolds, 
which is closely related to the Strominger-Yau-Zaslow conjecture for hyperK\"ahler manifolds. 
See \cite{ac} for recent related topics and \cite{cop} for non-algebraic cases.

\begin{thm}[{=Corollary \ref{cor-hk}}]\label{cor-hk-intro}Let $X$ be a $4$-dimensional projective hyperK\"ahler manifold and $L$ be a line bundle admitting a $($singular$)$ metric whose curvature 
is semi-positive and Lelong number is identically zero.  
Then $L$ is semi-ample.
\end{thm}

We summarize the contents of this paper. 
In Section \ref{preliminaries}, 
we collect the basic definitions and facts needed later. 
In Section \ref{proof of injectivity}, we prove 
the injectivity theorem (Theorem \ref{intro-main-inj}). 
In Section \ref{applications-inj}, we give 
direct applications  of the injectivity theorem to 
the extension theorem (Theorem \ref{intro-ext} and Corollary \ref{ext-cor}). 
Section \ref{applications} is devoted to results for semi-ampleness 
related to the abundance conjecture (Theorem \ref{main-thm} and Theorem  \ref{cor-hk-intro}).

\begin{ack}
The authors wish to express their deep gratitude to Professor Osamu Fujino for pointing out several mistakes in the draft and discussions, 
and to Professor Mihai P\u{a}un for a discussion about Theorem \ref{ext}. 
The first author thanks Professors Paolo Cascini, Christopher Hacon,  Keiji Oguiso, Vladimir Lazi\'c for discussions and comments.  
The first author is partially supported by Grant-in-Aid for Young Scientists (A) $\sharp$26707002 from JSPS and the Research expense from the JRF fund. 
The second author is partially supported by the Grant-in-Aid for 
Young Scientists (B) $\sharp$25800051 from JSPS.
\end{ack}

\section{Preliminaries}\label{preliminaries}
In this section, 
we first recall the definitions and basic properties 
of singular metrics and multiplier ideal sheaves. 
For more details, refer to \cite{Dem}, \cite{Dem-book}.

We denote by $X$ a compact complex manifold and 
by $F$ a line bundle on $X$ unless otherwise mentioned. 
Further we fix a smooth (hermitian) metric $g$ on $F$. 

\begin{defi}[Singular metrics and their curvature currents]\label{s-met}
\ \vspace{0.2cm} \\ 
(1) For an $L^{1}$-function $\varphi$ on $X$, 
the  metric $h$ defined by 
\begin{equation*}
h:= g e^{-\varphi} 
\end{equation*}
is called a \textit{singular metric} on $F$. 
Further $\varphi$ is called the \textit{weight} of $h$ 
with respect to the fixed smooth metric $g$. 
\vspace{0.2cm} \\ 
(2) 
A (singular) metric $h$ on $F$ is said to have 
{\textit{algebraic}} (resp. {\textit{analytic}}) {\textit{singularities}}, 
if there exists an ideal sheaf 
$\mathcal{I} \subseteq\mathcal{O}_{X}$ such that 
a weight $ \varphi$ of $h$ can be locally 
written as 
\begin{equation*}
\varphi = \frac{c}{2} \log \big( 
|f_{1}|^{2} + |f_{2}|^{2} + \cdots + |f_{k}|^{2}\big) +v, 
\end{equation*}
where $c$ is a positive rational number 
(resp. real number),  
$f_{1}, \dots f_{k}$ are local generators of $\mathcal{I}$,  
and $v$ is a smooth function. 
\vspace{0.2cm} \\ 
(3) The {\textit{curvature current}} 
$\sqrt{-1} \Theta_{h}(F)$ 
associated to $h$ is defined by   
\begin{equation*}
\sqrt{-1} \Theta_{h}(F) = \sqrt{-1} \Theta_{g}(F)+ \ddbar \varphi, 
\end{equation*}
where $  \sqrt{-1} \Theta_{g}(F)$ is 
the Chern curvature of $g$. 
\end{defi}
In this paper, we simply abbreviate 
the singular metric (resp. the curvature current)
to the metric (resp. the curvature). 
The Levi form $\ddbar \varphi$ is taken in the sense of 
distributions, 
and thus the curvature is a $(1,1)$-current 
but not always a smooth $(1,1)$-form. 
The curvature $\sqrt{-1} \Theta_{h}(F)$ of $h$ 
is said to be {\textit{semi-positive}} if 
$\sqrt{-1} \Theta_{h}(F) \geq 0$ in the sense of currents. 
%Remark that positivity of currents does not necessarily  
%mean strictly positive. 

\begin{defi}[Multiplier ideal sheaves]\label{multiplier}
Let $h$ be a metric on $F$ such that 
$\sqrt{-1} \Theta_{h}(F) \geq \gamma$  
for some smooth $(1,1)$-form $\gamma$ on $X$. 
The ideal sheaf $\I{h}$ defined 
to be  
\begin{equation*}
\I{h}(B):= \I{\varphi}(B):= \{f \in \mathcal{O}_{X}(B)\ \big|\ 
|f|\, e^{-\varphi} \in L^{2}_{\rm{loc}}(B) \}
\end{equation*}
for every open set $B \subseteq X$, is called 
the \textit{multiplier ideal sheaf} associated to $h$. 
\end{defi}
It is known that 
multiplier ideal sheaves are coherent sheaves. 
The following is a typical example of singular metrics 
that often appear in algebraic geometry.

\begin{eg}
For given holomorphic sections $\{s_{i}\}_{i=1}^{N}$ of 
the $m$-th tensor powers $F^{m}$ of $F$, 
the (singular) metric  $g e^{-\varphi}$ can be defined by 
\begin{align*}
\varphi:= \frac{1}{2m} \sum_{i=1}^{N} 
\log |s_{i}|^{2}_{g^{m}}.  
\end{align*}
Then the curvature of this metric is semi-positive, 
and further the multiplier ideal sheaf 
can be algebraically computed (see \cite{Dem}). 
In particular, for an effective $\mathbb{R}$-divisor 
$D$ on $X$, 
the metric $h_{D}$ on $\mathcal{O}_{X}(D)$ can 
be constructed by the natural section of $D$. 
We can easily check $\I{h_{D}}=\mathcal{O}_{X}(-\lfloor D \rfloor )$ if $D$ is a simple normal crossing divisor. 

\end{eg}

%The following restriction theorem can be obtained as 
%an application of the Ohsawa-Takegoshi extension theorem. 

%\begin{thm}[Restriction theorem]\label{res}
%Let $\varphi$ be a $($quasi-$)$plurisubharmonic function 
%$($psh function for short$)$ 
%on an open subset $B$ in $\mathbb{C}^n$. 
%For a submanifold $W$ on $B$, 
%we have 
%$\I{\varphi|_{W}} \subseteq \I{\varphi}|_{W}$.
%
%\end{thm}
%
We recall the definition of the Lelong number of singular metrics and Skoda's lemma 
which gives a relation between  
the multiplier ideal sheaf
and the Lelong number of singular metrics.

\begin{defi}[Lelong numbers]
Let $\varphi$ be a (quasi-)psh function  
on an open set $B$ in $\mathbb{C}^n$. 
The {\textit{Lelong number}} $\nu(\varphi, x)$ of $\varphi$ at $x \in B$ is defined by 
$$\nu(\varphi, x)=\liminf_{z \to x}\frac{\varphi(z)}{\mathrm{log}|z-x|}.$$
For a singular metric $h$ on $F$ such that 
$\sqrt{-1} \Theta_{h}(F) \geq \gamma$ 
for some smooth $(1,1)$-form, 
we define the Lelong number $\nu(h, x)$ of $h$ at $x \in X$ 
by $\nu(h, x):=\nu(\varphi, x)$, 
where $\varphi$ is a weight of $h$. 
\end{defi}

\begin{thm}[Skoda's lemma]\label{Skoda}
Let $\varphi$ be a $($quasi$)$-psh function on an open set $B$ in  $\mathbb{C}^n$. 
\begin{itemize}
\item[$\bullet$] If $\nu(\varphi, x)<1$, then we have 
$\I{\varphi}_{x}= \mathcal{O}_{B, x}$. 
\item[$\bullet$] If $\nu(\varphi, x) \geq n+s$ for 
some integer $s \geq 0$, then we have 
$\I{\varphi}_{x}\subseteq\mathfrak{M}_{B, x}^{s+1}$, 
where $\mathfrak{M}_{B, x}$ is the maximal ideal of 
$\mathcal{O}_{B, x}$.
\end{itemize} 
\end{thm}

%The following is an affirmative answer for 
%the openness conjecture (see \cite{Ber13}, \cite{GZ13})
%
%
%\begin{thm}[{The openness conjecture, \cite{Ber13}, \cite{GZ13}}]\label{open conj}
%
%Let $\varphi$ be a $($quasi$)$-psh function and $f$ be 
%a holomorphic function on an open set $B$ in  $\mathbb{C}^n$.
%If $|f|e^{-\varphi}$ is locally $L^{2}$-integrable, 
%then $|f|e^{-(1+\delta)\varphi}$ is also locally 
%$L^{2}$-integrable for a sufficiently small $\delta$. 
%Namely, $\I{\varphi}$ agrees with $\I{(1+\delta)\varphi}$. 
%\end{thm}

Next we give the definition of singularities of pairs.

\begin{defi}[Singularities of pairs] 
Let $X$ be a normal variety and $\Delta$ be an effective $\mathbb{Q}$-divisor 
on $X$ such that $K_X+\Delta$ is $\mathbb{Q}$-Cartier. 
Let $\varphi:Y\rightarrow X$ be a log resolution of $(X,\Delta)$. 
We set $$K_Y=\varphi^*(K_X+\Delta)+\sum a_iE_i,$$ where $E_i$ is a 
prime divisor on $Y$ for every $i$.
The pair $(X,\Delta)$ is called 
\begin{itemize}
\item[$\bullet$] \emph{kawamata log terminal} $($\emph{klt}, 
for short$)$ if $a_i > -1$ for all $i$, 
\item[$\bullet$]\emph{log canonical} $($\emph{lc}, for short$)$ if $a_i \geq -1$ for all $i$.
\end{itemize}

%Let $(X, \Delta)$ be an lc pair. 
%If there is a log resolution $\varphi:Y\to X$ of $(X, \Delta)$ such that 
%$\Exc (\varphi)$ is a divisor and that 
%5$a_i>-1$ for every $\varphi$-exceptional 
%divisor $E_i$, then 
%the pair $(X, \Delta)$ is called 
%{\em{divisorial log terminal}} ({\em{dlt}}, for short). 

%Let  $E$ be a prime divisor {\em{over}} $X$. 
%Then $a(E, X, \Delta)$ denotes the {\em{discrepancy coefficient}} 
%of $E$ with respect to $(X, \Delta)$. 
\end{defi}

\begin{defi}[Semi-log canonical, {\cite[Definition 1.1]{fujino-abundance}}]\label{slc} Let $X$ be a reduced $S_2$-scheme. We assume that it is pure $n$-dimensional and is normal crossing in codimension $1$. Let $\Delta$ be an effective $\mathbb{Q}$-Weil divisor on $X$ such that $K_X+\Delta$ is $\mathbb{Q}$-Cartier. 

Let $X=\bigcup X_i$ be the decomposition into irreducible components, and $\nu:X':=\coprod X'_i \rightarrow X=\bigcup X_i$ the {\em normalization}, where the \emph{normalization} $\nu : X'=\coprod X'_i \rightarrow X=\bigcup X_i$ means that $\nu|_{X_i'} : X'_i \to X_i$ is the usual normalization for any $i$. We call $X$ a \emph{normal scheme} if $\nu$ is isomorphic. Define the $\mathbb{Q}$-divisor $\Theta$ on $X'$ by $K_{X'}+\Theta=\nu^*(K_X+\Delta)$ and set $\Theta_i=\Theta|_{X_i'}$. 

We say that $(X,\Delta)$ is \emph{semi-log canonical} (for short, \emph{slc}) if $(X_i',\Theta_i)$ is an lc pair for every $i$. 
 \end{defi}

\section{A Version of the Injectivity theorem}\label{proof of injectivity}
The purpose of this section is to give 
an analytic version of  
the injectivity theorem by using multiplier ideal sheaves. 
This theorem is a generalization of \cite[Theorem 1.3]{Mat-inj} 
and it is applied in order to obtain 
the extension theorem (Theorem \ref{intro-ext}). 
See \cite{Fuj12-A} and \cite{Mat-inj}
for relations of various injectivity theorems and vanishing theorems.

\begin{thm}[A version of the injectivity theorem]
\label{main-inj}
Let $(F, h_{F})$ and $(L, h_{L})$ 
be $($singular$)$ hermitian line bundles 
with semi-positive curvature 
on a compact K\"ahler manifold $X$. 
Assume that there exists 
an effective $\mathbb{R}$-divisor $\Delta$ with 
\begin{align*}
h_{F}=h_{L}^{a}\cdot h_{\Delta},  
\end{align*}
where $a$ is a positive real number and  $h_{\Delta}$  is 
the singular metric defined by $\Delta$.

Then for a $($non-zero$)$ section $s$ of $L$
satisfying $\sup_{X} |s|_{h_{L}} < \infty$, 
the multiplication map 
%induced by the tensor product with $s$ 
\begin{equation*}
\Phi_{s}: H^{q}(X, K_{X} \otimes F \otimes \I{h_{F}}) 
\xrightarrow{\otimes s} 
H^{q}(X, K_{X} \otimes F\otimes L \otimes \I{h_{F} h_{L}})
\end{equation*}
is $($well-defined and$)$ injective for any $q$. 
Here $\I{h}$ denotes the multiplier ideal sheaf 
associated to a singular metric $h$.  
\end{thm}

\begin{rem}\label{rem-inj}
$(1)$ The multiplication map is well-defined 
thanks to the assumption of $\sup_{X} |s|_{h_{L}} < \infty$. 
When $h_{L}$ is a metric with minimal singularities on $L$, this assumption is always satisfied
for any section $s$ of $L$
$($see \cite{Dem} for the definition of metrics with minimal singularities$)$.
\vspace{0.1cm}\\
(2) The case of $\Delta=0$  
corresponds to the main result in \cite{Mat-inj}. 
To obtain the extension theorem $($Theorem \ref{intro-ext}$)$, 
it is important to consider the case of $\Delta \not=0$.
\vspace{0.1cm}\\
$(3)$ If $h_{L}$ and $h_{F}$ are smooth on a Zariski open set, 
the same conclusion holds under the weaker assumption of 
$\sqrt{-1}\Theta_{h_{F}}(F) \geq a \sqrt{-1}\Theta_{h_{L}}(L)$ 
$($see \cite{Fuj12-A}, \cite{Mat-Nad}$)$. 
\end{rem}

\begin{proof}
The proof is a generalization of the proof of the main result 
in \cite{Mat-inj} which corresponds to the case of $\Delta=0$. 
First of all, we recall Enoki's techniques to generalize  
Kollar's injectivity theorem, which give a proof of 
the special case where $h_{L}$ is smooth and $\Delta=0$. 
In this case, 
the cohomology group $H^{q}(X, K_{X} \otimes F)$ 
is isomorphic to the space 
of the harmonic forms with respect to $h_{F}$
\begin{equation*}
\mathcal{H}^{n, q}(F)_{h_{F}}:= \{u \mid u 
\text{ is a smooth $F$-valued $(n,q)$-form on } X 
\text{ such that } \dbar u  = D^{''*}_{h_{F}} u =0.  \}, 
\end{equation*}
where  $D^{''*}_{h_{F}}$ is the adjoint operator of 
the $\dbar$-operator. 
For an arbitrary harmonic form 
$u \in \mathcal{H}^{n, q}(F)_{h_{F}}$,  
we can conclude that $D^{''*}_{h_{F}h_{L}} su =0$  
thanks to semi-positivity of the curvature 
and $h_{F}=h^{a}_{L}$.
This step strongly depends on semi-positivity 
of the curvature.  
Then the multiplication map $\Phi_{s}$ 
induces the map from $\mathcal{H}^{n, q}(F)_{h_{F}}$ to 
$\mathcal{H}^{n, q}(F\otimes L)_{h_{F}h_{L}}$, and thus 
the injectivity is obvious.

%When $h$ has algebraic singularities, 
%the cohomology group 
%$H^{q}(X, K_{X} \otimes F \otimes \I{h})$ 
%can be represented by the space of harmonic forms on 
%a Zariski open set where $h$ is smooth. 
%Therefore we can give a similar proof as Enoki's proof,  
%thanks to semi-positivity of the curvature. 

In our situation, we must consider   
singular metrics with transcendental 
(non-algebraic) singularities. 
It is quite difficult 
to directly handle transcendental singularities, 
and thus in Step 1, we approximate a given metric $h_{F}$ by metrics 
$\{h_{\e} \}_{\e>0}$ that are smooth on a Zariski open set. 
%In this step, we fix the notation to apply the theory of 
%harmonic integrals and explain the sketch of the proof. 
Then we represent a given cohomology class in 
$H^{q}(X, K_{X} \otimes F \otimes \I{h_{F}})$ 
by the associated harmonic form  
$u_{\e}$ with respect to $h_{\e}$ on the Zariski open set. 
We want to show that $s u_{\e}$ is also harmonic 
by using the same method as Enoki's proof. 
However, the same argument as Enoki's proof fails  
since the curvature of $h_{\e}$ is not semi-positive. 
For this reason, in Step2, 
we investigate the asymptotic behavior of 
the harmonic forms $u_{\e}$ with respect to  
a family of the regularized metrics 
$\{ h_{\e} \}_{\e>0}$.  
Then we show that the $L^{2}$-norm 
$\| D^{''*}_{ h_{\e} h_{L, \e} }su_{\e}\|$ converges to zero as letting $\e$ go to zero, 
where $h_{L, \e}$ is a suitable approximation of $h_{L}$. 
Moreover, In Step 3, we construct solutions $\gamma_{\e}$ of 
the $\dbar$-equation $\dbar \gamma_{\e} = su_{\e}$ 
such that   
the $L^{2}$-norm $\| \gamma_{\e} \|$  
is uniformly bounded,  
by applying the $\rm{\check{C}}$ech complex with 
the topology induced by the local $L^{2}$-norms. 
In Step 4, we see  
\begin{equation*}
\| su_{\e} \| ^{2} = 
\lla su_{\e}, \dbar\gamma_{\e} \rra
\leq \| D^{''*}_{h_{\e} h_{L, \e}} su_{\e}\|
\| \gamma_{\e} \| \to 0 \quad {\text{as }} \e \to 0.
\end{equation*}
From these observations, we can conclude  that 
$u_{\e}$ converges to zero in a suitable sense, 
which completes the proof. 
\vspace{0.2cm} \\
{\bf{Step 1 (The equisingular approximation of $h_{F}$)}}
\vspace{0.1cm} \\
Throughout the proof, we fix a K\"ahler form $\omega$ on $X$. 
For the proof, 
we want to apply the theory of harmonic integrals, 
but the metric $h_{F}$ may not be smooth. 
For this reason, we approximate $h_{F}$ by metrics 
$\{ h_{\e} \}_{\e>0}$ that are smooth on a Zariski open set. 
By \cite[Theorem 2.3]{dps01}, we can obtain  
metrics $\{ h_{\e} \}_{\e>0}$ on $F$ with the following properties:
\begin{itemize}
\item[(a)] $h_{\e}$ is smooth on $X \setminus Z_{\e}$, where $Z_{\e}$ is a subvariety on $X$.
\item[(b)]$h_{\e_{2}} \leq h_{\e_{1}} \leq h_{F}$ 
holds for any 
$0< \e_{1} < \e_{2} $.
\item[(c)]$\I{h_{F}}= \I{h_{\e}}$.
\item[(d)]$\sqrt{-1} \Theta_{h_{\e}}(F) \geq -\e \omega$. 
\end{itemize}
Since the point-wise norm $|s|_{h_{L}}$ is bounded on $X$ 
and $h_{F}=h^{a}_{L}h_{\Delta}$, 
the set 
$\{x \in X \mid \nu(h_{F}, x)>0 \}$ 
is contained in the subvariety $Z$ defined by 
$Z:=s^{-1}(0) \cup {\rm{Supp}}\, \Delta$.   
Therefore  
we may assume a stronger property 
than property (a) (for example see \cite[Theorem 2.3]{Mat-inj}), namely
\begin{itemize}
\item[(e)] $h_{\e}$ is smooth on $Y:=X \setminus Z$, 
where $Z=s^{-1}(0) \cup {\rm{Supp}}\, \Delta$. 
\end{itemize}

Now we construct a \lq \lq complete" K\"ahler form on $Y$ 
with suitable potential function. 
Take a quasi-psh function $\psi$ on $X$ such that 
$\psi$ has a logarithmic pole along $Z$ 
and $\psi$ is smooth on $Y$.
Since quasi-psh functions are upper semi-continuous, 
the function $\psi$ is bounded above, 
and thus we may assume $\psi \leq -e$.
Then we define the $(1,1)$-form $\ome$ on $Y$ by 
\begin{equation*}
\ome:= \ell \omega + \ddbar \Psi, 
\end{equation*}
where 
$\ell$ is a positive number and $\Psi:=\frac{1}{\log(-\psi)}$. 
We can show that    
the $(1,1)$-form $\ome$ satisfies the following properties 
for a sufficiently large $\ell >0$: 
\begin{itemize}
\item[(A)] $\ome$ is a complete K\"ahler form on $Y$.
\item[(B)] $\Psi$ is bounded on $X$. 
\item[(C)] $\ome \geq \omega $.
\end{itemize} 
Indeed, properties (B), (C) are obvious 
by the definition of $\Psi$ and $\ome$, and property (A) follows 
from straightforward computations    
(see \cite[Lemma 3.1]{Fuj12-A} for the precise proof of property (A)).

In the proof, 
we mainly consider harmonic forms on $Y$ 
with respect to $h_{\e}$ and $\ome$. 
Let $L_{(2)}^{n, q}(Y, F)_{h_{\e}, \ome}$ be   
the space of the $L^{2}$-integrable $F$-valued $(n,q)$-forms 
$\alpha$ with respect to the inner product $\|\cdot \|_{h_\e, \ome}$
 defined by 
\begin{equation*}
\|\alpha \|^{2}_{h_\e, \ome}:= \int_{Y} 
|\alpha |^{2}_{h_{\e}, \ome}\ \ome^{n}. 
\end{equation*}
Then we can obtain the following orthogonal decomposition: 
\begin{equation*}
L_{(2)}^{n, q}(Y, F)_{h_{\e}, \ome}
=
{\rm{Im}}\,\dbar
\oplus
\mathcal{H}^{n, q}(F)_{h_{\e}, \ome}
\oplus {\rm{Im}}\,D^{''*}_{h_{\e}}.   
\end{equation*}
Here the operator $D^{'*}_{h_{\e}}$ 
(resp. $D^{''*}_{h_{\e}}$) denotes   
the closed extension of the formal adjoint of the 
$(1,0)$-part $D^{'}_{h_{\e}}$ (resp. $(0,1)$-part $D^{''}_{h_{\e}}=\dbar$) 
of the Chern connection $D_{h_{\e}}=D^{'}_{h_{\e}}+ D^{''}_{h_{\e}}$.  
Note that they coincide with 
the Hilbert space adjoints since $\ome$ is complete.  
Further $\mathcal{H}^{n, q}(F)_{h_{\e}, \ome}$ denotes   
the space of the harmonic forms with respect to 
$h_{\e}$ and $\ome$, namely 
\begin{equation*}
\mathcal{H}^{n, q}(F)_{h_{\e}, \ome}:= 
\{\alpha   \mid \alpha 
\text{ is an } F\text{-valued } (n,q)\text{-form with  }
\dbar \alpha= D^{''*}_{h_{\e}}\alpha=0.    \}. 
\end{equation*}
Harmonic forms in $ \mathcal{H}^{n, q}(F)_{h_{\e}, \ome}$ 
are smooth by the regularization theorem for elliptic operators. 
These results are known to specialists.  
The precise proof of them   
can be found in \cite[Claim 1]{Fuj12-A}.

For every $(n, q)$-form $\beta$, 
we have 
$|\beta|^{2}_{\ome}\ \ome^{n} \leq 
|\beta|^{2}_{\omega}\ \omega^{n}$ since  
the inequality $\ome \geq \omega$ holds by property (C).  
From this inequality and property (b) of $h_{\e}$,
we obtain 
\begin{equation}\label{ine}
\|\alpha \|_{h_{\e}, \ome} \leq 
\|\alpha \|_{h_{\e}, \omega} \leq 
\|\alpha \|_{h_{F}, \omega}
\end{equation}
for an $F$-valued $(n,q)$-form $\alpha$, 
which plays a crucial role in the proof.

Take an arbitrary cohomology class
$\{u \} \in H^{q}(X, K_{X} \otimes F \otimes \I{h_{F}})$ 
represented by an $F$-valued 
$(n, q)$-form $u$ with $\|u \|_{h_{F}, \omega} < \infty$. 
In order to prove that 
the multiplication map $\Phi_{s}$ is injective, 
we assume that the cohomology class of $su$ 
is zero in 
$H^{q}(X, K_{X}\otimes F\otimes L \otimes \I{h_{F}h_{L}})$. 
Our final goal is to show that 
the cohomology class of $u$ is actually zero 
under this assumption.

By inequality (\ref{ine}), 
we have $\|u\|_{h_{\e}, \ome} < \infty$ for any $\e > 0$.  
Therefore by the above orthogonal decomposition, 
there exist  
$u_{\e} \in \mathcal{H}^{n, q}(F)_{h_{\e}, \ome}$ and 
$v_{\e} \in L_{(2)}^{n,q-1}(Y, F)_{h_{\e}, \ome}$ such that 
\begin{equation*}
u=u_{\e}+\dbar v_{\e}. 
\end{equation*}
Note that  
the component of ${\rm{Im}} D^{''*}_{h_{\e}}$ is zero 
since $u$ is $\dbar$-closed.

At the end of this step, we explain  
the strategy of the proof. 
In Step 2, we show that 
$\|D^{''*}_{h_{\e} h_{L, \e}} 
s u_{\e} \|_{ h_{\e} h_{L, \e}, \ome}$ 
converges to zero 
as letting $\e$ go to zero. 
Here $h_{L, \e}$ is the singular metric on $L$ defined by 
\begin{align*}
h_{L, \e}:= h_{\e}^{1/a}\, h_{\Delta}^{-1/a}. 
\end{align*}
Since the cohomology class of $su$ is zero, 
there are solutions $\gamma_{\e}$ of the $\dbar$-equation 
$\dbar \gamma_{\e} = s u_{\e}$. 
For the proof, we need to obtain $L^{2}$-estimates of them. 
In Step 3, 
we construct solutions $\gamma_{\e}$ of the $\dbar$-equation   
$\dbar \gamma_{\e} = s u_{\e}$ such that  
the norm $\| \gamma_{\e} \|_{h_{\e} h_{L, \e}, \ome}$ 
is uniformly bounded. 
Then we have 
\begin{equation*}
\|su_{\e} \|^{2}_{h_{\e} h_{L, \e}, \ome} \leq 
\|D^{''*}_{h_{\e} h_{L, \e}} s u_{\e} \|_{h_{\e} h_{L, \e}, \ome}
\| \gamma_{\e} \|_{h_{\e} h_{L, \e}, \ome}.
\end{equation*} 
By Step 2 and Step 3, we can conclude that  
the right hand side goes to zero as 
letting $\e$ go to zero. 
In Step 4, from this convergence, we prove that 
$u_{\e}$ converges to zero  
in a suitable sense, 
which implies that the cohomology class of $u$ is zero. 
\vspace{0.1cm}\\
\vspace{0.2cm} \\
{\bf{Step 2 (A generalization of Enoki's proof of the injectivity theorem)}}
\vspace{0.1cm}\\
The aim of this step is to prove the following proposition, 
whose proof can be seen as 
a generalization of Enoki's injectivity theorem. 
\vspace{-0.2cm}
\begin{prop} \label{D''}
As letting $\e$ go to zero, the norm 
$\|D^{''*}_{h_{\e} h_{L, \e}} s u_{\e} \|_{h_{\e} h_{L, \e}, \ome}$ converges  
to zero. 
\end{prop}
\hspace{-0.5cm}
\begin{proof}[Proof of Proposition \ref{D''}] 
We have the following inequality: 
\begin{equation}\label{ine2}
\|u_{\e} \|_{h_{\e}, \ome} 
\leq \|u \|_{h_{\e}, \ome} 
\leq \|u \|_{h, \omega}.  
\end{equation}
This inequality is often used in the proof. 
The first inequality follows from the definition of  $u_{\e}$ 
and  the second inequality follows from 
inequality (\ref{ine}). 
Remark the right hand side does not depend on $\e$. 
By applying Nakano's identity and the density lemma 
to $u_{\e}$ (for example see \cite[Proposition 2.4]{Mat-inj}), 
we obtain 
\begin{equation}\label{B-eq}
0 = \lla \sqrt{-1}\Theta_{h_{\e}}(F)
\Lambda_{\ome} u_{\e}, u_{\e}
  \rra_{h_{\e}, \ome} +
\|D^{'*}_{h_{\e}}u_{\e} \|^{2}_{h_{\e}, \ome}. 
\end{equation}
Note that the left hand side is zero since $u_{\e}$ is harmonic.
Let $A_{\e}$ be the first term and $B_{\e}$ be 
the second term of the right hand side of equality (\ref{B-eq}). 
First,  
we show that the first term $A_{\e}$ and 
the second term $B_{\e}$ converge to zero. 
For simplicity, we denote the integrand of $A_{\e}$ by $g_{\e}$, 
namely 
\begin{equation*}
g_{\e}:= \langle  \sqrt{-1}\Theta_{h_{\e}}(F)
\Lambda_{\ome} u_{\e}, u_{\e}
 \rangle_{h_{\e}, \ome}. 
\end{equation*}
Then there exists a positive constant $C>0$ (independent of $\e$) 
such that
\begin{equation}\label{ine3}
g_{\e} \geq -\e C |u_{\e}|^{2}_{h_{\e}, \ome}. 
\end{equation}
It is easy to check this inequality. 
Indeed, let $\lambda_{1}^{\e} \leq \lambda_{2}^{\e} \leq 
\dots \leq \lambda_{n}^{\e} $ be the 
eigenvalues of $\sqrt{-1}\Theta_{h_{\e}}(F)$ with respect to 
$\ome$. 
Then for any point $y \in Y$ there exists 
a local coordinate $(z_{1}, z_{2}, \dots, z_{n})$ 
centered at $y$ such that 
\begin{align*}
\sqrt{-1}\Theta_{h_{\e}}(F) = \sum_{j=1}^{n} 
\lambda_{j}^{\e} dz_{j} \wedge d\overline{z_{j}}\quad \text{and} \quad 
\ome = \sum_{j=1}^{n} 
 dz_{j} \wedge d\overline{z_{j}}
\quad {\rm{ at}}\ y. 
\end{align*}
When we locally write $u_{\e}$ as 
$u_{\e} =\sum_{|K|=q} f_{K}^{\e}\ dz_{1}\wedge \dots \wedge dz_{n} 
\wedge d\overline{z}_{K}$, 
we have  
\begin{equation*}
g_{\e}= \sum_{|K|=q} 
\Big{(} \sum_{j \in K} \lambda_{j}^{\e} \Big{)} 
|f_{K}^{\e}|^{2}_{h_{\e}} 
\end{equation*}
by a straightforward computation.  
On the other hand, from 
property (C) of $\ome$ and property (d) of $h_{\e}$, 
we have
$\sqrt{-1}\Theta_{h_{\e}}(F) 
\geq -\e \omega 
\geq -\e \ome$. 
This implies $\lambda_{j}^{\e} \geq -\e$, 
and thus we obtain inequality (\ref{ine3}).

From equality (\ref{B-eq}) and inequality (\ref{ine3}), 
we obtain 
\begin{align*}
0 \geq A_{\e} &= \int_{Y} g_{\e}\ \ome^{n} \\
& \geq -\e C \int_{Y} |u_{\e}|^{2}_{h_{\e}, \ome}\ \ome^{n}\\
& \geq -\e C \|u \|^{2}_{h_{F}, \omega}. 
\end{align*}
The last inequality follows from inequality (\ref{ine2}). 
Therefore $A_{\e}$ converges to zero, and further   
we can conclude that $B_{\e}$ also converges to zero 
by equality (\ref{B-eq}).

To apply Nakano's identity to $su_{\e}$ again, 
we first check $su_{\e} \in 
L_{(2)}^{n,q}(Y, F \otimes L)_{h_{\e} h_{L, \e}, \ome}$. 
By the assumption, 
the point-wise norm $|s|_{h_{L}}$ with respect to $h_{L}$ is 
bounded, and further we have 
$|s|_{h_{L, \e}} \leq |s|_{h_{L}}$
from property (b) of $h_{\e}$. 
They imply   
\begin{equation*}
\|s u_{\e} \|_{h_{\e} h_{L, \e}, \ome} \leq 
\sup_{X} |s|_{h_{L, \e}}  \|u_{\e} \|_{h_{\e}, \ome}  \leq
\sup_{X} |s|_{h_{L}}  \|u \|_{h_{F}, \omega} < \infty. 
\end{equation*}
Remark that the right hand side does not depend on $\e$. 
By applying Nakano's identity to $su_{\e}$, 
we obtain 
\begin{align}  
& \|D^{''*}_{h_{\e} h_{L, \e}}
su_{\e} \|^{2}_{h_{\e} h_{L, \e}, \ome} \nonumber \\ 
=& \lla \sqrt{-1}\Theta_{h_{\e} h_{L, \e}}(F\otimes L)
\Lambda_{\ome} su_{\e}, su_{\e}
  \rra_{h_{\e} h_{L, \e}, \ome} +
\|D^{'*}_{h_{\e} h_{L, \e}}su_{\e} \|^{2}_{h_{\e} h_{L, \e}, \ome}\label{B-eq2}.
\end{align}
Here we used $\dbar s u_{\e}=0$. 
We first see that  
the second term of the right hand side converges to zero. 
Since $s$ is a holomorphic $(0, 0)$-form, we have 
$D^{'*}_{h_{\e} h_{L, \e}}su_{\e} =
s D^{'*}_{h_{\e}}u_{\e}$. 
Thus we have 
\begin{equation*}
\|D^{'*}_{h_{\e} h_{L, \e}}su_{\e} \|^{2}_{h_{\e} h_{L, \e}, \ome} \leq
\sup_{X}|s|^{2}_{h_{L, \e}}  \int_{Y}
|D^{'*}_{h_{\e}}u_{\e} |^{2}_{h_{\e}, \ome}\ \ome^{n} 
\leq \sup_{X}|s|^{2}_{h_{L}} B_{\e}. 
\end{equation*}
Since $|s|^{2}_{h_{L}}$ is bounded and 
$B_{\e}$ converges to zero, 
the second term 
$\|D^{'*}_{h_{\e} h_{L, \e}}su_{\e} \|_{h_{\e} h_{L, \e}, \ome}$ also converges to zero.

For the proof of the proposition, 
it remains to show that 
the first term of the right hand side of equality (\ref{B-eq2})
converges to zero. 
We can easily see  
$\sqrt{-1}\Theta_{h_{\e} h_{L, \e}}(F\otimes L) = {(1+1/a)}\sqrt{-1}\Theta_{h_{\e}}(F)$
from $\sqrt{-1} \Theta_{h_{\Delta}}=0$ on $Y$ and 
the definition of $h_{L, \e}$. 
Therefore we obtain 
\begin{equation*}
\lla \sqrt{-1}\Theta_{h_{\e} h_{L, \e}}(F\otimes L)
\Lambda_{\ome} su_{\e}, su_{\e}
  \rra_{h_{\e} h_{L, \e}, \ome} =
(1+1/a) \int_{Y} |s|^{2}_{h_{L, \e}} 
g_{\e}\ \ome^{n}
\end{equation*}
Now 
we investigate $A_{\e}$ in details. 
By the definition of $A_{\e}$, we have
\begin{equation*}
A_{\e}= \int_{\{ g_{\e} \geq 0 \}} g_{\e}\ \ome^{n} + 
\int_{\{ g_{\e} \leq 0 \}} g_{\e}\ \ome^{n}. 
\end{equation*}
It is easy to see that 
the second term converges to zero. 
Indeed, simple computations and inequality (\ref{ine3}) 
imply 
\begin{align*}
0 \geq \int_{\{ g_{\e} \leq 0 \}} g_{\e}\ \ome^{n} &\geq 
-\e C \int_{\{ g_{\e} \leq 0 \}}|u_{\e}|^{2}_{h_{\e}, \ome}\ 
\ome^{n} \\
&\geq -\e C \int_{Y}|u_{\e}|^{2}_{h_{\e}, \ome}\ 
\ome^{n} \\
&\geq -\e C \| u \|^{2} _{h_{F}, \omega}. 
\end{align*}
Therefore the first term also converges to zero. 
Now we have  
\begin{align*}
\int_{Y} |s|^{2}_{h_{L, \e}} 
g_{\e}\ \ome^{n} =  \Big\{ \int_{\{ g_{\e} \geq 0 \}} 
|s|^{2}_{h_{L, \e}} g_{\e}\ \ome^{n} + 
\int_{\{ g_{\e} \leq 0 \}} 
|s|^{2}_{h_{L, \e}} g_{\e}\ \ome^{n}\Big\}.
\end{align*}
On the other hand, we have the inequalities  
\begin{align*}
\bullet \hspace{0.3cm} 0 \leq \int_{\{ g_{\e} \geq 0 \}} 
|s|^{2}_{h_{L, \e}} g_{\e}\ \ome^{n} &\leq  \sup_{X}|s|^{2}_{h_{L, \e}} 
\int_{\{ g_{\e} \geq 0 \}} 
g_{\e}\ \ome^{n} \\ 
& \leq   \sup_{X}|s|^{2}_{h_{L}}\ \int_{\{ g_{\e} \geq 0 \}} g_{\e}\ \ome^{n},\\
\bullet \hspace{0.3cm}  0 \geq \int_{\{ g_{\e} \leq 0 \}} 
|s|^{2}_{h_{L, \e}} g_{\e}\ \ome^{n}_{\e} &\geq   
\sup_{X}|s|^{2}_{h_{L, \e}} \int_{\{ g_{\e} \leq 0 \}} 
 g_{\e}\ \ome^{n} \\
&\geq   \sup_{X}|s|^{2}_{h_{L}} \int_{\{ g_{\e} \leq 0 \}} g_{\e}\ \ome^{n}. 
\end{align*}
Therefore the right hand side of equality (\ref{B-eq2}) 
converges to zero. 
We obtain the conclusion of Proposition \ref{D''}. 
\end{proof}
\ 
\vspace{0.2cm} \\
{\bf{Step 3 (A construction of 
solutions of the $\dbar$-equation via 
the $\bf\rm{\check{C}}$eck complex)}}
\vspace{0.1cm} \\
By the absolutely same method as \cite[Step 3]{Mat-inj}, 
we can prove the following proposition. 
See \cite[Step 3]{Mat-inj} for the proof. 

\begin{prop}\label{sol-1}
There exist $F$-valued $(n, q-1)$-forms $\alpha_{\e}$ 
on $Y$ 
with the following properties\,$:$ 
\begin{equation*}
{\rm{(1)}}\hspace{0.2cm} \dbar \alpha_{\e}=u-u_{\e}.   
\quad 
{\rm{(2)}}\hspace{0.2cm} \text{The norm }
\|\alpha_{\e} \|_{h_{\e}, \ome} 
 \text{ is uniformly bounded}. 
\end{equation*}
\end{prop}
\hspace{0.1cm}\\ \ 
{\bf{Step 4 (The limit of the harmonic forms)}}
\vspace{0.1cm}\\ \ 
In this step, we investigate the limit of $u_{\e}$ and 
complete the proof of Theorem \ref{main-inj}. 
First we prove the following proposition. 
\begin{prop}\label{sol}
There exist $F\otimes L$-valued $(n, q-1)$-forms 
$\gamma_{\e}$ 
on $Y $  
with the following properties\,$:$ 
\begin{equation*}
{\rm{(1)}}\hspace{0.2cm} \dbar \gamma_{\e}=su_{\e}.   
\quad 
{\rm{(2)}}\hspace{0.2cm} \text{The norm }
\|\gamma_{\e} \|_{h_{\e} h_{L, \e}, \ome} 
 \text{ is uniformly bounded}. 
\end{equation*}
\end{prop}
\begin{proof}
There exists an $F\otimes L$-valued $(n, q-1)$-form $\gamma$ 
such that $\dbar \gamma =  su$ and 
$\|\gamma \|_{h_{F} h_{L}, \omega} < \infty$.  
(Recall that we are assuming that the cohomology class of $su$ is zero in 
$H^{q}(X, K_{X}\otimes F\otimes L \otimes 
\I{h_{F} h_{L}})$.) 
If we take $\alpha_{\e}$ with the properties in Proposition \ref{sol-1} 
and put $\gamma_{\e}:= -s \alpha_{\e} + \gamma$,  
then we have $\dbar \gamma_{\e} = su_{\e}$. 
Further an easy computation yields 
\begin{align*}
\|\gamma_{\e}\|_{h_{\e} h_{L, \e}, \ome} &\leq 
\|s \alpha_{\e}\|_{h_{\e} h_{L, \e}, \ome} + 
\|\gamma \|_{h_{\e} h_{L, \e}, \ome}\\
&\leq \sup_{X}|s|_{h_{L}} \|\alpha_{\e} \|_{h_{\e}, \ome} + 
\|\gamma \|_{h_{F}h_{L}, \ome}. 
\end{align*}
Since $\| \gamma \|_{h_{F}h_{L}, \ome}
\leq \| \gamma \|_{h_{F}h_{L}, \omega} < \infty$ and the norm $\|\alpha_{\e} \|_{h_{\e}, \ome}$ is uniformly bounded, 
the right hand side can be estimated by a constant independent of $\e$. 
\end{proof}
We consider the limit of the norm 
$\| s u_{\e}\|_{h_{\e}h_{L, \e}, \ome}$.

\begin{prop}\label{converge}
The norm $\| s u_{\e}\|_{h_{\e}h_{L, \e}, \ome}$ 
converges to zero as letting $\e$ go to zero. 
\end{prop}
\begin{proof}
By taking    
$\gamma_{\e} \in L_{(2)}^{n, q-1}
(Y, F \otimes L)_{h_{\e}h_{L, \e}, \ome}$ satisfying the properties 
in Proposition \ref{sol},   
we obtain 
\begin{align*}
\| s u_{\e}\|_{h_{\e}h_{L, \e}, \ome}^{2}
&=\lla  s u_{\e}, \dbar \gamma_{\e}  
 \rra_{h_{\e}h_{L, \e}, \ome} \\
&=
 \lla  D^{''*}_{h_{\e}h_{L, \e}}s u_{\e},\gamma_{\e}  
 \rra_{h_{\e}h_{L, \e}, \ome} \\
&\leq 
\|D^{''*}_{h_{\e}h_{L, \e}} s u_{\e}\| _{h_{\e}h_{L, \e}, \ome}
\|\gamma_{\e} \|_{h_{\e}h_{L, \e}, \ome}. 
\end{align*}
By Proposition \ref{sol}, the norm   
$\|\gamma_{\e} \|_{h_{\e}h_{L, \e}, \ome}$ is uniformly bounded.
On the other hand, the norm 
$\|D^{''*}_{h_{\e}h_{L, \e}} s u_{\e}\| _{h_{\e}h_{L, \e}, \ome}$ 
converges to zero by Proposition \ref{D''}. 
Therefore the norm $\| s u_{\e}\|_{h_{\e}, \ome}$ also converges to zero. 
\end{proof}

Fix a small number $\e_{0}>0$. 
Then for any positive number $\e$ with $0< \e < \e_{0}$, 
by property (b) of $h_{\e}$, we obtain
\begin{equation*}
\|u_{\e} \|_{h_{\e_{0}}, \ome} \leq \|u_{\e} \|_{h_{\e}, \ome} 
\leq  \|u \|_{h_{F}, \omega}. 
\end{equation*}
It says that the norm of $u_{\e} $ with respect to 
$h_{\e_{0}}$ is 
uniformly bounded. 
Therefore 
there exists a subsequence of $\{ u_{\e} \}_{\e >0}$ 
that  converges to 
$\alpha \in L_{(2)}^{n, q}(Y, F)_{h_{\e_{0}}, \ome}$
with respect to the weak $L^{2}$-topology. 
For simplicity, we denote this subsequence 
by the same notation $\{ u_{\e} \}_{\e >0}$. 
Then we prove the following proposition.

\begin{prop}\label{zero}
The weak limit $\alpha$ of $\{ u_{\e} \}_{\e >0}$ 
in $L_{(2)}^{n, q}(Y, F)_{h_{\e_{0}}, \ome}$
is zero. 
\end{prop}
\begin{proof}

For every positive number $\delta>0$, 
we define the open subset $A_{\delta}$ of $Y$ by 
$A_{\delta}:= \{x \in Y \mid  |s|^{2}_{h_{L,\e_{0}}} > \delta  \}$. 
By an easy computation, we have 
\begin{align*}
\| s u_{\e}  \|^{2}_{h_{\e}h_{L, \e}, \ome} 
&\geq \| s u_{\e}  \|^{2}_{h_{\e_{0}}h_{L, \e_{0}}, \ome} \\
&\geq \int_{A_{\delta}} |s|^{2}_{ h_{L, \e_{0}} } 
|u_{\e}|^{2}_{h_{\e_{0}}, \ome}\ \ome^{n} \\
&\geq \delta  \int_{A_{\delta}} 
|u_{\e}|^{2}_{h_{\e_{0}}, \ome}\ \ome^{n} \geq 0
\end{align*}
for any $\delta>0$. 
Since the left hand side converges to zero, 
the norm $\|u_{\e} \|_{h_{\e_{0}}, \ome, A_{\delta}}$ on $A_{\delta}$ 
also converges to zero. 
Notice that $u_{\e} |_{A_{\delta}}$ 
converges to $\alpha |_{A_{\delta}}$ with respect to 
the weak $L^{2}$-topology in 
$ L_{(2)}^{n, q}(A_{\delta}, F)_{h_{\e_{0}}, \ome}$. 
Here  $u_{\e} |_{A_{\delta}}$ (resp. $\alpha |_{A_{\delta}}$) denotes 
the restriction of  $u_{\e}$ (resp. $\alpha$) to $A_{\delta}$. 
Indeed  
for every 
$\gamma \in L_{(2)}^{n, q}(A_{\delta}, F)_{h_{\e_{0}}, \ome}$, 
the inner product  
$\lla u_{\e} |_{A_{\delta}}, \gamma
\rra_{A_{\delta}} 
= \lla u_{\e} , \widetilde{\gamma} 
\rra_{Y} $
converges to 
$\lla \alpha, \widetilde{\gamma}
\rra_{Y} 
= \lla \alpha|_{A_{\delta}} , \gamma  
\rra_{A_{\delta}} $, where $\widetilde{\gamma}$ denotes the zero extension of $\gamma$ to 
$Y$. 
Since $u_{\e} |_{A_{\delta}}$ 
converges to $\alpha |_{A_{\delta}}$, 
we obtain
\begin{equation*}
\|\alpha |_{A_{\delta}} \|_{h_{\e_{0}}, \ome, A_{\delta}} 
\leq 
\liminf_{\e \to 0}\|u_{\e} |_{A_{\delta}} \|_{h_{\e_{0}},  \ome, 
A_{\delta}}=0. 
\end{equation*}
(Recall the norm of the weak limit can be 
estimated by 
the limit inferior of the norms of sequences.)
Therefore 
we have $\alpha |_{A_{\delta}} = 0$ for any $\delta>0$. 
By the definition of $A_{\delta}$, the union of 
$\{A_{\delta} \}_{\delta >0}$ is equal to $Y=X \setminus Z$, which asserts 
that the weak limit $\alpha $ is zero on $Y$.
\end{proof}

By using Proposition \ref{zero}, 
we complete the proof of Theorem \ref{main-inj}. 
By the definition of $u_{\e}$, 
we have 
\begin{equation*}
u = u_{\e} + \dbar v_{\e}. 
\end{equation*}
Proposition \ref{zero} says  
that $\dbar v_{\e}$ converges to $u$ with respect to 
the weak $L^{2}$-topology. 
Then it is easy to see that 
$u$ is a $\dbar$-exact form 
(that is, $u \in {\rm{Im}} \,\dbar$). 
This is because
the subspace ${\rm{Im}} \dbar$ 
is closed in 
$L_{(2)}^{n, q}(Y, F)_{h_{\e_{0}}, \ome}$ 
with respect to the weak $L^{2}$-topology. 
Indeed, for every 
$\gamma = \gamma_{1} + D^{''*}_{h_{\e_{0}}}\gamma_{2} \in
\mathcal{H}^{n, q}(F)_{h_{\e_{0}}, \ome}
\oplus {\rm{Im}}\,D^{''*}_{h_{\e_{0}}}$, we have 
$\lla u, \gamma  \rra =
\lim_{\e \to 0}\lla \dbar v_{\e}, \gamma_{1} + D^{''*}_{h_{\e_{0}}}\gamma_{2} 
\rra =0$.
Therefore we can conclude $u \in {\rm{Im}}\,\dbar$.

In summary, we proved that $u$ is a $\dbar$-exact form in 
$L_{(2)}^{n, q}(Y, F)_{h_{\e_{0}}, \ome}$, which  
says that the cohomology class $\{u \}$ of $u$ is zero
in $H^{q}(X, K_{X} \otimes F \otimes \I{h_{\e_{0}}})$. 
By property (c), we obtain the conclusion of Theorem \ref{main-inj}.

\end{proof}

\section{Proof of Corollaries related to the Extension conjecture}
\label{applications-inj}
The purpose of this section is to obtain  
some extension theorems as applications of our injectivity theorem. 
For this purpose, 
by making use of the injectivity theorem (Theorem \ref{main-inj}), 
we first prove the following extension theorem,  
which can be seen as a special case of 
the extension conjecture for dlt pairs.

\begin{thm}\label{ext}
Let $X$ be a compact K\"ahler manifold and 
$\Delta=S+B$ be an effective $\mathbb Q$-divisor 
with the following assumptions: 
\begin{itemize}
\item[(1)] $\Delta$ is a simple normal crossing divisor with 
$0 \leq \Delta \leq 1$ and $\lfloor \Delta \rfloor =S$. 
\item[(2)] 
$K_{X}+\Delta$ is $\mathbb Q$-linearly equivalent to 
an effective divisor $D$ with $S\subset\Supp\,D$. 
\item[(3)] $K_{X}+\Delta$ admits a $($singular$)$ 
metric $h$ with semi-positive curvature. 
\end{itemize}
Then for an integer $m \geq 2$ with Cartier divisor $m(K_{X}+S+B)$ 
and any section $u \in H^{0}(S, \mathcal{O}_{S}(m(K_{X}+S+B)))$ 
that belongs to the image of 
$H^{0}(S, \mathcal{O}_{S}(m(K_{X}+S+B))\otimes \I{h^{m-1}h_{B}} ) \to 
H^{0}(S, \mathcal{O}_{S}(m(K_{X}+S+B))) $,  
the section $u$ can be extended to a section in $H^{0}(X, \mathcal{O}_{X}(m(K_{X}+S+B)) )$.

\vspace{0.2cm}
Moreover if we assume that $h \leq C h_{D}$ for some $C>0$,  
where $h_{D}$ is the singular metric induced by $D$, 
then any cohomology class  $u \in 
H^{q}(S, \mathcal{O}_{S}(m(K_{X}+\Delta))\otimes \I{h^{m-1}h_{B}} )$ 
can be extended to a class in 
$H^{q}(X, \mathcal{O}_{X}(m(K_{X}+\Delta))\otimes \I{h^{m-1}h_{B}} )$ for any $q \geq 0$.

\end{thm}

\begin{rem}
$(1)$ If $X$ is projective and $S$ is smooth   
$($namely $(X,\Delta)$ is plt$)$ 
and further if the divisor induced by a give 
$u \in H^{0}(S, \mathcal{O}_{S}(m(K_{X}+S+B)))$ 
is larger than 
the extension obstruction divisor 
$($which is zero if $K_{X}+\Delta$ is nef$)$, 
then $K_{X}+\Delta$ admits a $($singular$)$ metric $h$ with $\sup_{S} |u|_{h} < \infty$. 
$($see \cite[Corollary 1.8]{dhp}$)$. 
%\vspace{0.1cm}\\
%$(2)$ By the restriction theorem, $\sup_{S} |u|_{h} < \infty$  implies 
%$u \in H^{0}(S, \mathcal{O}_{S}(m(K_{X}+S+B))\otimes \I{h^{m-1}h_{B}} )$ 
%if $S$ is smooth. 
\vspace{0.1cm}\\
$(2)$ Even if $S$ has singularities,
we can construct a $($singular$)$ metric $h$ with $\sup_{S} |u|_{h} < \infty$  
in the special case where  
$u$ is zero along the singular locus of $S$,   
by the same method as in \cite{dhp}.

\end{rem}

\begin{proof}
We may add 
the additional assumption of  
$h \leq h_{D}$, 
where $h_{D}$ is the singular metric on $\mathcal{O}_{X}(K_{X}+\Delta)$ 
defined by the effective divisor $D$. 
Indeed, for a smooth metric $g$ on $\mathcal{O}_{X}(K_{X}+\Delta)$ and 
an $L^{1}$-function $\varphi$ (resp. $\varphi_{D}$)  
with $h= g\, e^{-\varphi} $ (resp. $h_{D}= g\, e^{-\varphi_{D}} $),  
the metric defined by $g\, e^{-\max( \varphi, \varphi_{D})}$
satisfies assumption $(3)$ again.

We prove only the first conclusion since  
the second conclusion follows from the same argument as the first conclusion. 
We put $G:=m(K_{X}+ \Delta)$, 
and consider the following exact sequence: 
\begin{equation*}
0 \rightarrow \mathcal{O}_{X}(G - S) \otimes I (h^{m-1}h_{B}) 
\rightarrow \mathcal{O}_{X}(G ) \otimes \I{h^{m-1}h_{B}}
\rightarrow \mathcal{O}_{S}(G ) \otimes \I{h^{m-1}h_{B}} 
\rightarrow0.
\end{equation*}
We prove 
the induced homomorphism
\begin{equation*}
H^{q}(X, \mathcal{O}_{X}(G - S) \otimes I (h^{m-1}h_{B})) 
\to 
H^{q}(X, \mathcal{O}_{X}(G) \otimes I (h^{m-1}h_{B})) 
\end{equation*}
is injective by the injectivity theorem. 
Then the conclusion follows from the induced long exact sequence.

By the assumption on the support of $D$, we can take 
an integer $a>0$ such that $aD$ is a 
Cartier divisor and $S \leq a D$. 
Then we have the following commutative diagram:

\[\xymatrix{
& \hspace{-3.5cm} H^{q}(X, \mathcal{O}_{X}(G) \otimes I (h^{m-1}h_{B})) \supseteq \mathrm{Im}\,(+S)   
 \ar[d]^{+ (aD -S)}   
& \\
H^{q}(X, \mathcal{O}_{X}(G - S) \otimes I (h^{m-1}h_{B}))  
\ar[ru]^{+ S}   
\ar[r]_{\hspace{-1.0cm}+ aD}  & 
H^{q}(X, \mathcal{O}_{X}(G - S + aD) \otimes I (h^{a+m-1}h_{B})), & 
}\]
where the map $+S:H^{q}(X, \mathcal{O}_{X}(G - S) \otimes I (h^{m-1}h_{B}))  \to H^{q}(X, \mathcal{O}_{X}(G ) \otimes I (h^{m-1}h_{B})).  $
Our goal is to show 
that the map to the upper right is injective. 
For this goal, we show that the horizontal map is injective 
as an application of Theorem \ref{main-inj}.

By the definition of $G$, 
we have 
\begin{equation*}
G -S =m(K_{X}+\Delta)-S=K_{X} + (m-1)(K_{X}+ \Delta) + B. 
\end{equation*}
Then the line bundle $F:= \mathcal{O}_{X}((m-1)(K_{X}+ \Delta) + B)$ 
equipped with the metric $h_{F}:= h^{m-1} h_{B}$
and the line bundle $L:=\mathcal{O}_{X}(aD)$ 
equipped with the metric $h_{L}:= h^{{a}}$ 
satisfy the assumptions in Theorem \ref{main-inj}. 
Indeed, we have $h_{F}=h_{L}^{(m-1)/a} h_{B}$
by the construction, and further 
the point-wise norm $|s_{aD}|_{h_{L}}$ is bounded on $X$ 
by the inequality $h \leq h_{D}$, 
where $s_{aD}$ is the natural section of $aD$.    
Therefore the horizontal map is injective 
by Theorem \ref{main-inj}. 
\end{proof}

To obtain some results related to the abundance conjecture
(Theorem \ref{main-thm-semiample} and Corollary \ref{sm semi positive}), 
we need the following corollary, 
which is a slight generalization of Theorem \ref{ext}.

\begin{cor}\label{ext-cor2}
Under the same situation as in Theorem \ref{ext}, 
instead of assumption $(3)$, 
we assume the following assumption: 
\begin{itemize}
\item[$(3')$] There exist effective $\mathbb{Q}$-divisors 
$E$ and $F$ and a $($singular$)$ metric $h$ on $\mathcal{O}_{X}(F)$ 
with semi-positive curvature such that 
\begin{itemize}
\item[$\bullet$] $K_{X}+\Delta \sim_{\mathbb{Q}} E+F$,  
\item[$\bullet$] $E+B$ is simple normal crossing and $E$ has no common component with $S$, 
\item[$\bullet$] $\nu(h, x) = 0$ at every point $x \in S$. 
\end{itemize}
\end{itemize}
Let $\widetilde{s}$ be a section on $X$ with ${\rm{div}}\, \widetilde{s} = mE$. 
Then for a section $u \in H^{0}(S, \mathcal{O}_{S}(mF))$, 
the section $u \cdot \widetilde{s} \in 
H^{0}(S, \mathcal{O}_{S}(m(K_{X}+\Delta)))$ can be extended to 
a section in $H^{0}(X, \mathcal{O}_{X}(m(K_{X}+\Delta)))$. 
\end{cor}

\begin{proof}
Let $h_{E}$ be the singular metric on $\mathcal{O}_{X}(E)$ 
induced by the section $\widetilde{s} \in H^{0}(X,\mathcal{O}_{X}(mE))$. 
By the definition, the metric $h_{E}$ satisfies 
$\sqrt{-1}\Theta_{h_E} \geq 0$ and $\sup |\widetilde{s}|_{h^{m}_{E}} < \infty$. 
The product $h \cdot  h_{E}$ determines the singular metric 
on $K_{X}+\Delta$ with semi-positive curvature. 
Therefore it is sufficient to show that 
$u \cdot \widetilde{s}$ belongs to 
$H^{0}(S, \mathcal{O}_{S}(m(K_{X}+\Delta))
\otimes  \mathcal{J})$, 
where we put $\mathcal{J}:= \I{ h^{m-1}h^{m-1}_{E}h_{B} }$ for simplicity.   

In the first step, 
we see that 
\begin{equation*}
\mathcal{J}_{x} = \I{h^{m-1}_{E}h_{B}}_{x} 
\end{equation*}
for every $x \in S$, 
where $\mathcal{F}_{x}$ denotes the stalk of a sheaf $\mathcal{F}$ at $x$. 
Let $f$ be a holomorphic function 
on an open neighborhood $U_{x}$ of $x \in S$ with $ f \in \I{h^{m-1}_{E}h_{B}}_{x} $, 
and let $\varphi$ (resp. $\varphi_{E}$, $\varphi_{B}$) be 
a local weight of $h$ (resp. $h_{E}$, $h_{B}$). 
By taking a real number $p>1$ 
with $\I{h^{p(m-1)}_{E}h^{p}_{B}}=\I{h^{m-1}_{E}h_{B}}$, 
we may assume that $|f| e^{ -p(m-1)\varphi_{E} - p\varphi_{B}}$ 
is $L^{2}$-integrable on $U_{x}$. 
Then,  
by taking the positive number $q$ with $1/p + 1/q = 1$, 
we obtain 
\begin{align*}
\int_{U_{x}} |f|^{2} 
e^{-2(m-1)\varphi -2(m-1)\varphi_{E} - 2\varphi_{B}} 
&\leq 
\Big(\int_{U_{x}} |f|^{2p} 
e^{ -2p(m-1)\varphi_{E} - 2p\varphi_{B}} \Big)^{1/p}\cdot  
\Big(\int_{U_{x}} e^{-2q (m-1)\varphi} \Big)^{1/q}\\
\end{align*}
by H\"older's inequality. 
The function $e^{-2q (m-1)\varphi}$ is locally $L^{2}$-integrable 
for any $q>0$ 
by Skoda's lemma and the assumption of the Lelong number.
On the other hand, as mentioned above,  
$|f|^{p} e^{ -p(m-1)\varphi_{E} - p\varphi_{B}}$ is also locally $L^{2}$-integrable. 
Therefore we have $\mathcal{J}_{x} = \I{h^{m-1}_{E}h_{B}}_{x} $
for every $x \in S$.

In the second step, 
we prove 
\begin{equation*}
u\cdot  \widetilde{s} \in H^{0}(S, \mathcal{O}_{S}(m(K_{X}+\Delta))
\otimes \mathcal{J}|_{S}), 
\end{equation*}
where $\mathcal{J}|_{S}$ is the restriction of $\mathcal{J}$ defined by   
\begin{equation*}
\mathcal{J}|_{S} := 
\mathcal{J}\cdot \mathcal{O}_{S}= 
\mathcal{J}/(\mathcal{J}\cap \mathcal{I}_{S}).
\end{equation*}
Let $\widetilde{u}$ be a local extension of $u$ 
on an open neighborhood $U_{x}$ of $x \in S$. 
By the klt condition of $B$, 
we can take a real number $p>1$ 
with $\I{h^{p}_{B}}=\mathcal{O}_{X}$. 
Then for the holomorphic function $g:=\widetilde{u}\cdot  \widetilde{s}$, 
by taking the positive number $q$ with $1/p + 1/q = 1$, 
we obtain
\begin{align*}
\int_{U_{x}} |g|^{2} 
e^{-2(m-1)\varphi_{E} - 2\varphi_{B}} 
&\leq 
\Big(\int_{U_{x}}  |g|^{2p} e^{ -2p(m-1)\varphi_{E}- 2p\varphi_{B}} \Big)^{1/p}\cdot   
\Big(\int_{U_{x}} 1 \Big)^{1/q} \\
&\leq 
\sup_{U_{x}} |g|^{2p} e^{ -2p(m-1)\varphi_{E}}
\Big(\int_{U_{x}}  e^{- 2p\varphi_{B}} \Big)^{1/p}\cdot   
\Big(\int_{U_{x}} 1 \Big)^{1/q}
\end{align*}
by H\"older's inequality again.
The point-wise norm 
$|g|^{2p} e^{ -2p(m-1)\varphi_{E}}$ is bounded by the choice of $h_{E}$, 
and 
$e^{-2q (m-1)\varphi}$ is locally $L^{2}$-integrable 
for any $q>0$ 
by the assumption on the Lelong number.
It implies that 
$u\cdot  \widetilde{s} $ (locally) belongs to $\I{h^{m-1}_{E}h_{B}}|_{S}=
\mathcal{}\mathcal{J}|_{S}$.

Finally we show  
\begin{align*}
u \cdot \widetilde{s} &\in H^{0}(S, \mathcal{O}_{S}(m(K_{X}+\Delta))\otimes \mathcal{J} ). 
%\\
%&=H^{0}(S, \mathcal{O}_{X}(m(K_{X}+\Delta))\otimes 
%\mathcal{J}/(\mathcal{J} \cdot \mathcal{I}_{S}) ). 
\end{align*}
By simple computations  
we have $\mathcal{O}_{S}\otimes \mathcal{J}= 
\mathcal{O}_{X}\otimes \mathcal{J}/(\mathcal{J} \cdot \mathcal{I}_{S})$, 
and thus, by the second step, it is sufficient to see  
$$\mathcal{J}\cap \mathcal{I}_{S} 
= \mathcal{J} \cdot \mathcal{I}_{S}.$$
Here $\mathcal{I}_{S}$ denotes the ideal sheaf defined by $S$. 
By the first step and the assumption on the support of $E+B$. 
we have  
\begin{equation*}
\mathcal{J}_{x} = \I{h^{m-1}_{E}h_{B}}_{x} =\mathcal{O}_{X}(-\lfloor(m-1)E+B \rfloor )
\end{equation*}
for every $x \in S$. 
Since the divisors 
$S$ and $E+B$ has no common component by the assumption on $E$,  
we can easily see  
$\mathcal{J}\cap \mathcal{I}_{S} 
= \mathcal{J} \cdot \mathcal{I}_{S}.$
Therefore $u \cdot \widetilde{s}$ actually belongs to 
$H^{0}(S, \mathcal{O}_{S}(m(K_{X}+\Delta))
\otimes \mathcal{J})$.
The conclusion follows from Theorem \ref{ext}. 
\end{proof}

\begin{rem}
When we apply the injectivity theorem in order to extend sections, 
we need to handle $\mathcal{O}_{S} \otimes \I{\varphi }$ $($not $\I{\varphi }|_{S}$$)$. 
When we apply the Ohsawa-Takegoshi extension theorem, 
we usually use the restriction of multiplier ideal sheaves $\I{\varphi }|_{S} $.
It is relatively difficult to handle $\mathcal{O}_{S} \otimes \I{\varphi }$. 
However the support condition of $E$ $($the second assumption of the above corollary$)$ 
fortunately appears in the proof of the applications related to the abundance conjecture, 
which asserts $\mathcal{O}_{S} \otimes \I{\varphi }= \I{\varphi }|_{S}$. 
\end{rem}

In a special case of the above corollary when $E=\mathcal{O}_X$ and $\widetilde{s}=1 \in H^0(X, \mathcal{O}_X)$, we obtain the following:

\begin{cor}\label{ext-cor}
Under the same situation as in Theorem \ref{ext}, 
instead of assumption $(3)$, 
we assume the following assumption: 
\begin{itemize}
\item[$(3'')$] $K_{X}+\Delta$ admits a $($singular$)$ 
metric $h$ such that $\sqrt{-1}\Theta_{h} \geq 0$ and 
$\nu(h, x) = 0$ at every point $x \in S$. 
\end{itemize}
Then for an integer $m \geq 2$ 
with Cartier divisor $m(K_{X}+\Delta)$,  
a section $u \in 
H^{0}(S, \mathcal{O}_{S}(m(K_{X}+\Delta)) )$ 
can be extended to a section in 
$H^{0}(X, \mathcal{O}_{X}(m(K_{X}+\Delta)) )$. 
\end{cor}
%\begin{proof}
%By Theorem \ref{ext}, it is sufficient to show  
%$ \I{h^{m-1}h_{B}}_{x} = \mathcal{O}_{X, x}$ for every point $x \in S$. 
%Let $\varphi$ (resp. $\varphi_{B}$) be 
%a local weight of $h$ (resp. $h_{B}$) 
%on an open neighborhood $U_{x}$ of $x \in S$. 
%By the klt condition of $B$, 
%we can take a real number $p>1$ with $\I{p \varphi_{B}}=\mathcal{O}_{X}$. 
%By taking the integer $q$ with $1/p + 1/q = 1$, 
%we obtain 
%\begin{align*}
%\int_{U_{x}} 
%e^{-2(m-1)\varphi  - 2\varphi_{B}} 
%&\leq 
%\Big(\int_{U_{x}} 
%e^{ - 2p\varphi_{B}} \Big)^{1/p}\cdot  
%\Big(\int_{U_{x}} e^{-2q (m-1)\varphi} \Big)^{1/q}
%\end{align*}
%by H\"older's inequality. 
%By the assumption on the Lelong number of $h$ and the choice of $p$, 
%the right hand side is finite for a sufficiently small $U_{x}$. 
%It implies $ \I{h^{m-1}h_{B}}_{x} = \mathcal{O}_{X, x}$ 
%for every point $x \in S$.
%\end{proof}

For further applications of the above theorems, 
we prepare the following lemma.

\begin{lem}\label{pull-min}
Let $\varphi$ be a $($quasi$)$-psh function 
on a complex manifold $X$.  
If the Lelong number $\nu(\varphi, x_{0})$ is zero at $x_{0} \in X$, 
then for any modification 
$\pi: Y \to X$, the Lelong number 
$\nu(\pi^{*} \varphi, y)$ is zero at every point $y \in \pi^{-1}(x_{0})$. 
\end{lem}
\begin{proof}
For a contradiction, we assume that 
$\nu(\pi^{*} \varphi, y_{0})>0$ 
for some point $y_{0} \in \pi^{-1}(x_{0})$.
By Skoda's lemma (see Theorem \ref{Skoda}), 
we can take a sufficiently large number $m>0$ such that 
$\pi^{*}dV_{X} e^{-2m \pi^{*}\varphi}$ is 
not integrable on a neighborhood of $y_{0}$, 
where $dV_{X}$ is a standard volume form on a neighborhood 
$U$ of $x_{0}$. 
By the change of variable formula,  
we have  
\begin{align*}
\int_{U} e^{-2m \varphi}dV_{X} = 
\int_{\pi^{-1}(U)}  e^{-2m \pi^{*}\varphi} \pi^{*}dV_{X}. 
\end{align*}
By the assumption of $\nu(\varphi, x_{0})=0$, 
the left hand side if finite for a sufficiently small $U$. 
It is a contradiction to the choice of $m$. 
Therefore $\nu(\pi^{*} \varphi, y)=0$ at every point $y \in \pi^{-1}(x_{0})$.

\end{proof}

\section{Proof of Corollaries related to the abundance conjecture}
\label{applications}

In this section, 
we prove some applications related to the abundance conjecture. The proof of the following theorem is based on \cite[Section 8]{dhp} and \cite[Theorem 5.9]{fg3}. 
We use the different MMP from them to preserve metric conditions. 

\begin{thm}\label{main-thm-semiample}Assume that Conjecture \ref{abun conj} holds in dimension $n-1$.
Let $X$ be an $n$-dimensional normal projective variety and $\Delta$ be an $\mathbb{Q}$-divisor 
with the following assumptions: 
\begin{itemize}
\item $(X,\Delta)$ is a klt pair such that there exists an effective $\mathbb{Q}$-divisor $D$ such that $K_X+\Delta \sim_{\mathbb{Q}}D$.
 
\item  There exists a projective birational morphism $\varphi:Y \to X$ such that $Y$ is smooth and $\varphi^*(\mathcal{O}_X(m(K_X+\Delta)))$ admits 
a $($singular$)$ metric whose curvature 
is semi-positive and Lelong number is identically zero on $\mathrm{Supp}\,\varphi^*D$.
Here $m$ is a positive integer such that $m(K_X+\Delta)$ is Cartier.  
\end{itemize}
Then $K_X+\Delta$ is semi-ample.

\end{thm}

\begin{proof}Note that Conjecture \ref{abun conj}  in dimension $n-1$ implies that the existence of good minimal model for $(n-1)$-dimensional klt pairs by \cite[Theorem 4.3]{gl} or \cite[Remark 2.6]{dhp}. First we may assume that $\kappa(K_X+\Delta)=0$ by Kawamata's theorem \cite[Theorem 7.3]{kawamata_pluri} (see also \cite[5.6 Lemma]{kemamc}).
There exists the effective $\mathbb{Q}$-divisor $D$ such  that
$$D \sim_{\mathbb{Q}}K_X+\Delta$$
by the assumption. For a contradiction, assume  $D \not =0$.
We may assume that $\varphi$ is a log resolution of $(X,\Delta+D)$.  We write 
$$K_Y+B=\varphi^*(K_X+\Delta)\sim_{\mathbb{Q}} \varphi^*D.
$$
Let $l= \mathrm{lct(\varphi^*D; Y, B)}$. Set effective $\mathbb{Q}$-divisors $C$, $S$, and  $G$ such that $C+S=(B+l\varphi^*D)^{\geq0}$, $\lfloor C+S\rfloor=S$, and $G=(B+l\varphi^*D)^{\leq0}$. Then there are no common components of any two divisors of $C$, $S$, and $G$, and we have   
$$K_Y+C+S\sim_{\mathbb{Q}} (1+l)\varphi^*D+G,
$$

$$S+C-G=B+l\varphi^*D.$$
 Note that $S\not =0$ and $G$ is $\varphi$-exceptional.
Then we see that $S \subseteq \mathrm{Supp}\,\varphi^*D$. Let us consider the MMP for $(Y,C+S)$ over $X$ by \cite{bchm} and \cite[Theorem 2.3]{fujino-ss}. Then this program contracts only the divisor $G$. Let $f:Y \dashrightarrow Y'$ be a minimal model of this program and $\varphi':Y' \to X$ the induced morphism. Then $S$ is not contracted by $f$ since $S$ and $G$ have no common component. Denote $C'$ and $S' (\not =0)$ by the strict transforms on $Y$ of $C$ and $S$. Then it follows that $$K_{Y'}+C'+S' \sim_{\mathbb{Q}} (1+l) \varphi'^*D.
$$
Thus $K_{Y'}+C'+S' $ is nef. 

\begin{cl}\label{claim-ex}For a sufficiently large and divisible  $m'  \in \mathbb{Z}$,
the restriction map 
$$H^0(Y', m'(K_{Y'}+C'+S' )) \to H^0(S', m'(K_{Y'}+C'+S' )) 
$$
 is surjective

\end{cl}
\begin{proof}[Proof of Claim \ref{claim-ex}]\label{Proof of claim-ex}
 Let $u$ be a non-zero section in $H^0(S', m'(K_{Y'}+C'+S' ))$ and let

\begin{equation*}
\xymatrix{ & W\ar[dl]_{\alpha} \ar[dr]^{\beta}\\
 Y\ar@{-->}[rr]_{f}  & & Y'}
\end{equation*}
be common log resolutions such that $\beta$ is isomorphic over the generic point of every lc center of  $(Y,' C'+S')$. Then we write 
$$K_{W}+S_W+F\sim_{\mathbb{Q}}(1+l)\beta^*\varphi'^*D+E,
$$
where $S_W$ is the strict transform of $S'$, and $E$ and $F$ are effective $\mathbb{Q}$-divisor such that $\lfloor F\rfloor=0$ and any two divisors of  $S_W$, $F$ and $E$ have no common component. Note that $m'F$ and $m'E$ are Cartier since $m'$ is sufficiently large and divisible. Let $u_{m'E}$ and $h_{m'E}$ be  the global section  and the metric associated to $m'E$ respectively. Then by the assumption we have a semi-positive singular metric $h$ on $m' (1+l)\beta^*\varphi'^*D$ induced by the pullback (we use the same symbol since it contains no confusion). 
%Thus we conclude that 

%$$\varphi'^*u\cdot u_{m'E} \in H^0(S_W, \mathcal{O}_W(m'(K_{W}+S_W+F)) \otimes \I{h^{\frac{m'-1}{m'}}h_{(m'-1)E}h_{F}} )
%$$
%So $h_{S_W}$ is well-defined and it satisfies that  $h |_{S_W} \leq {1}/{|u|^{1/m'}}$.   Thus we have 

%$$(h \otimes h_{m'E})_{|_{S_W}} \leq {1}/{|u\cdot u_{m'E}|^{1/m'}}.$$

By applying Corollary \ref{ext-cor2} for $F=(1+l)\beta^*\varphi'^*D$  (cf. Lemma\,\ref{pull-min}) and  $\widetilde{s}=u_{m'E}$,  we have a section $U \in H^0(m'(K_W+S_W+F))$ such that $U_{|_{S_W}}=\beta^*u \otimes (u_{m'E})|_{S_W} $. Thus we see that
$$H^0(Y', m'(K_{Y'}+C'+S' )) \to H^0(S', m'(K_{Y'}+C'+S' )) 
$$
 is surjective. Indeed we see that
 $$H^0(W, m'(K_W+S_W+F)) \to H^0(Y', m'(K_{Y'}+C'+S'))
 $$
  is isomorphic by mapping $s \mapsto \beta^*s \otimes u_{m'E}$ for a section $s \in H^0(W, m'(K_W+S_W+F)),$
  and $$H^0(S', m'(K_{Y'}+C'+S')) \to H^0(S_W, m'(K_{S_W}+F|_{S_W}))
 $$
 is injection by mapping  $t \mapsto \beta^*t \otimes (u_{m'E})|_{S_W}$ for a section $t \in H^0(S', m'(K_{Y'}+C'+S'))$ by $\beta_*\mathcal{O}_{S_W}=\mathcal{O}_{S'}$ from Koll\'ar--Shokurov's connectedness Theorem  \cite[17.4 Theorem]{KolFlip}.
We finish the proof of Claim \ref{claim-ex}.
   
\end{proof}
On the other hand, this restriction map is zero map since 
$$\kappa(K_X+\Delta)=\kappa(K_{Y'}+C'+S')=0$$
 and $S' \subseteq \mathrm{Supp}\,\varphi'^*D$.  The pair $(S', C'_{S'})$ is a projective semi-log canonical pair such that $K_{S'}+C'_{S'}$ is nef, where $(K_{Y'}+C'+S')|_{S'}=(K_{S'}+C'_{S'})$. This implies $K_{S'}+C'_{S'}$ is semi-ample, because of  the  abundance conjecture (Conjecture \ref{abun conj}) holds in dimension $n-1$ and \cite[Theorem 1.5]{fg3} or \cite{hx} (here we need the assumption of projectivity). This is a contradiction to Claim \ref{claim-ex}. Thus $D=0$. We finish the proof.
\end{proof}

%\begin{cor}\label{slc}Assume that Conjecture \ref{abun conj} holds in dimension $n-1$.
%Let $(X,\Delta$ be an $n$-dimensional projective semi-log canonical pair, where $\Delta$ is $\mathbb{Q}$-effective divisor with the following assumptions: 
%\begin{itemize}
%\item  There exists a projective birational morphism $\varphi:Y \to X$ on to each components  such that $Y$ is smooth and $\varphi^*(\mathcal{O}_X(m(K_X+\Delta)))$ admits 
%a $($singular$)$ metric whose curvature 
%is semi-positive and Lelong number is identically zero.
%Here $m$ is a positive integer such that $m(K_X+\Delta)$ is Cartier.  
%\end{itemize}
%If $\kappa(K_X+\Delta) \geq 0$, then $K_X+\Delta$ is semi-ample.

%\end{cor}

%\begin{proof}By \cite[Theorem 1.5]{fg3} or \cite{hx}, we may assume that $(X,\Delta)$ is log canonical since the normalization factors through  $\varphi$. By taking dlt blow up (Theorem \ref{dltblowup}), we may assume that $(X,\Delta)$ is dlt. By Theorem \ref{main-thm-semiample}, we may also assume that $S := \lfloor \Delta \rfloor  \not =0$ (cf. \cite[Proposition 2.41]{komo}). We 

%\end{proof}

By combining the abundance theorem in dimension $3$ (\cite[Theorem 1.1]{k1}, \cite[1.1 Theorem]{kemamc}, \cite[Theorem 0.1]{fujino-abundance}), we obtain the following results:

\begin{cor}\label{sm semi positive}Let $(X,\Delta)$ be a $4$-dimensional projective klt pair.  
Assume that there exists a projective birational morphism $\varphi:Y \to X$ on to each components such that $Y$ is smooth and $\varphi^*(\mathcal{O}_X(m(K_X+\Delta)))$ admits 
a singular metric whose curvature is semi-positive and Lelong number is identically zero 
(in particular it is satisfied if $h$ is smooth). 
Here $m$ is an integer with Cartier divisor $m(K_X+\Delta)$. 
If $\kappa(K_X+\Delta) \geq 0$, then $K_X+\Delta$ is semi-ample.

\end{cor}

\begin{rem}\label{rem-smooth}When $h$ is smooth,  we can show Corollary \ref{sm semi positive} without using our injectivity theorem. By replacing Theorem  \ref{intro-main-inj} to the generalized Enoki's injectivity theorem after Fujino \cite[Theorem 1.2 and Corollary1.3]{Fuj12-A} in the proof of Theorems \ref{ext}, \ref{main-thm-semiample},  and Corollary \ref{ext-cor2}, we can obtain it.

\end{rem}

Finally we give a result for semi-ampleness 
by combining with Verbitsky's non-vanishing theorem (\cite[Theorem 4.1]{ver}).  

\begin{cor}\label{cor-hk}Let $X$ be a $4$-dimensional projective hyperK\"ahler manifold  
and $L$ be a line bundle admitting a $($singular$)$ 
metric whose curvature is semi-positive and Lelong number is identically zero 
$($in particular it is satisfied if $h$ is smooth$)$. 
Then $L$ is semi-ample.

\end{cor}

\begin{proof} It is enough to show $\kappa(L) \geq 0$ by Corollary \ref{sm semi positive} since if there exists an effective $\mathbb{Q}$-divisor such that $D \sim_{\mathbb{Q}}L$,  the pair $(X, \varepsilon D)$ is  klt and $K_X+\varepsilon D \sim_{\mathbb{Q}}\varepsilon L$ for sufficiently small $\varepsilon>0$.  If $q(L,L)>0$, then $L$ is big, where $q(\cdot, \cdot)$ is the Bogomolov--Beauville--Fujiki form. In the case of $q(L,L)=L^{\mathrm{dim}\,X}=0$, then $\kappa(L) \geq 0$ follows from  \cite[Theorem 4.1]{ver}.   

\end{proof}


\begin{thebibliography}{elmnp}
\bibitem[A1]{ambro-qlog}
F.~Ambro, 
{\textit{Quasi-log varieties}}, 
Tr. Mat. Inst. Steklova {\textbf{240}} (2003), 
Biratsion. Geom. Linein. Sist. Konechno Porozhdennye Algebry, 220--239; translation in 
Proc. Steklov Inst. Math. 2003, no. 1 (240), 214--233 

\bibitem[A2]{ambro} 
F.~Ambro, 
{\textit{An injectivity theorem}}, 
to appear in Compos. Math. 

\bibitem[AC]{ac} 
E.~Amerik and F.~Campana, 
{\textit{Characteristic foliation on non-uniruled smooth divisors 
on projective hyperkaehler manifolds}}, 
Preprint, arXiv:1405.0539.




\bibitem[BCHM]{bchm}
C. Birkar, P. Cascini, C. D. Hacon, and J. $\mathrm{M^{c}}$Kernan, 
{\textit{Existence of minimal models for varieties of log general type}}, 
J. Amer. Math. Soc. {\bf{23}} (2010), 405--468.

\bibitem[COP]{cop}
F.~Campana, K.~Oguiso, and T.~Peternell, 
{\textit{Non-algebraic Hyperkaehler manifolds}}, 
J. Differential Geom. {\bf 85}, Number 3 (2010), 397--424.

\bibitem[Dem]{Dem}
J.-P. Demailly,   
{\textit {Analytic methods in algebraic geometry}}, 
Lecture Notes on the web page of the author.


\bibitem[Dem-book]{Dem-book}
J.-P. Demailly,   
{\textit {Complex analytic and differential geometry}}, 
Lecture Notes on the web page of the author.


\bibitem[DHP]{dhp}
J.-P.~Demailly, C.~D.~Hacon, and M.~P\u{a}un, 
{\textit{Extension theorems, non-vanishing and the existence of good minimal models}}, 
Acta Math. {\bf 210} (2013) 203--259.



\bibitem[DPS]{dps01}
J.-P. Demailly, T. Peternell, and M. Schneider,   
{\textit{Pseudo-effective line bundles on compact K\"ahler manifolds}}, 
International Journal of Math. {\bf{6}} (2001), 689--741. 




\bibitem[E]{Eno90}
I. Enoki. 
{\textit{Kawamata-Viehweg vanishing theorem for compact 
K\"ahler manifolds}}, 
Einstein metrics and Yang-Mills connections (Sanda, 1990), 59--68. 


\bibitem[EV-book]{ev} 
H.~Esnault and E.~Viehweg, 
{\textit{Lectures on vanishing theorems}},  
DMV Seminar, {\textbf{20}}, Birkh\"auser Verlag, Basel, (1992). 

\bibitem[F1]{fujino-abundance}
O.~Fujino, 
{\textit{Abundance theorem for semi log canonical threefolds}}, 
Duke Math. J. {\textbf{102}} (2000), no. 3, 513--532. 





\bibitem[F2]{fujino-book1} 
O.~Fujino, 
{\textit{Introduction to the log minimal model program for log canonical pairs}}, 
preprint (2008). 



\bibitem[F3]{fujino-ss} 
O.~Fujino, 
{\textit{Semi-stable minimal model program for varieties with trivial 
canonical divisor}}, 
Proc. Japan 
Acad. Ser. A Math. Sci. {\textbf{87}} (2011), no. 3, 25--30. 

\bibitem[F4]{F-fund}
O.~Fujino, 
{\textit{Fundamental theorems for the log minimal model program}}, 
Publ. Res. Inst. Math. Sci. {\textbf{47}} (2011), no. 3, 727--789.




\bibitem[F5]{Fuj12-A}
O. Fujino, 
{\textit{A transcendental approach to Koll\'ar's injectivity theorem}}, 
Osaka J. Math. {\bf{49}} (2012), no. 3, 833--852.



\bibitem[F-book]{fujino-book2} 
O.~Fujino, 
{\textit{Foundation of the  minimal model program}}, 
preprint (2014). 



\bibitem[FG]{fg3}
O. Fujino and Y. Gongyo, 
{\textit{Log pluricanonical representations and the abundance conjecture}}, 
Compositio Math.  {\bf{150}} (2014), 593--620

\bibitem[GL]{gl}
Y.~Gongyo and B.~Lehmann, 
{\textit{Reduction maps and minimal model theory}},  
Compositio Math. {\bf 149}, No.2, 295--308. 

%\bibitem[GR]{gr65}
%R. C. Gunning and H. Rossi. 
%\textit{Analytic functions of several complex variables.}
%Prentice-Hall, Inc., Englewood Cliffs, N.J. (1965).

%\bibitem[GZ]{GZ13}
%Q. Guan, X. Zhou. 
%\textit{Strong openness conjecture for plurisubharmonic functions.}
%Preprint, arXiv:1311.3781v1.

\bibitem[HX]{hx}
C.~Hacon and  C.~Xu, 
{\textit{On Finiteness of B-representations and Semi-log Canonical Abundance}}, 
Preprint, arXiv:1107.4149.

\bibitem[Ka]{kawamata_pluri} 
Y.~Kawamata, 
{\textit{Pluricanonical systems on minimal algebraic varieties}}, 
Invent. Math. {\textbf{79}} (1985), no. 3, 567--588.


\bibitem[Ka2]{k1}
Y.~Kawamata, 
{\textit{Abundance theorem for minimal threefolds}},   
Invent. Math.  {\textbf{108}}  (1992),  no. 2, 229--246.

\bibitem[KaMM]{kamama}
Y. Kawamata, K. Matsuda, and K. Matsuki, 
\textit{Introduction to the minimal model problem}, 
Algebraic geometry, Sendai, 1985,  283--360, Adv. Stud. Pure Math., 10, North-Holland, Amsterdam, 1987. 

\bibitem[KeMMc]{kemamc}
S. Keel, K. Matsuki, and J. $\mathrm{M^{c}}$Kernan, 
{\textit{Log abundance theorem for threefolds}},  
Duke Math. J.  {\textbf{75}}  (1994),  no. 1, 99--119, 
Corrections to: ``{\textit{Log abundance theorem for threefolds}}'',  Duke Math. J.  {\textbf{122}}  (2004),  no. 3, 625--630. 


\bibitem[Ko1]{Kol86}
J. Koll\'ar, 
\textit{Higher direct images of dualizing sheaves. I}, 
Ann. of Math. (2) {\bf{123}} (1986), no. 1, 11--42.

%\bibitem[KoKo]{kk}
%J. Koll\'ar and S. J Kov\'acs,  Log canonical singularities are Du Bois, J. Amer. Math. Soc.  {\textbf{23}} (2010), 791--813.

\bibitem[Ko2]{KolFlip}
J.~Koll\'ar, Adjunction and discrepancies, {\em{Flips and Abundance for algebraic threefolds}}, 
Ast\'erisque {\textbf{211}} (1992), 183--192. 

\bibitem[KoM]{komo}
J. Koll\'ar and S. Mori, 
{\textit{Birational geometry of algebraic varieties}}, 
Cambridge Tracts in Math.,134 (1998).




\bibitem[Mat1]{Mat-Nad}
S. Matsumura, 
{\textit{A Nadel vanishing theorem via injectivity theorems}},  
to appear in Math. Ann.  



\bibitem[Mat2]{Mat-inj}
S. Matsumura, 
{\textit{An injectivity theorem with multiplier ideal sheaves of singular metrics with transcendental singularities}},  
Preprint, arXiv:1308.2033v2. 

\bibitem[N]{nakayama-zariski-abun}
N. Nakayama, 
\textit{Zariski decomposition and abundance}, 
MSJ Memoirs, {\bf{14}}. Mathematical Society of Japan, Tokyo, 2004. 


\bibitem[O]{ohsawa} 
T.~Ohsawa, 
{\textit{
On a curvature condition that implies a cohomology injectivity 
theorem of Koll\'ar--Skoda type}}, 
Publ. Res. Inst. Math. Sci. {\textbf{41}} (2005), no. 3, 565--577.



\bibitem[T]{tan} 
S.~G.~Tankeev, 
{\textit{On $n$-dimensional canonically polarized varieties 
and varieties of fundamental type}}, 
Math. USSR-Izv. vol. {\textbf{5}} (1971), no. 1, 29--43.  


\bibitem[V]{ver}
M. Verbitsky,  
{\textit{HyperK$\mathrm{\ddot{a}}$hler SYZ conjecture and semipositive line bundles}}, 
Geom. Funct. Anal.  {\bf{19}} (2010), no. 5, 1481--1493.


\end{thebibliography}
\end{document}